\documentclass[12pt, letterpaper]{amsart}
\usepackage{amsfonts,amsthm,mathtools}
\usepackage[all,cmtip]{xy}
\usepackage{fullpage}
\usepackage{amsmath}
\usepackage{amssymb}
\usepackage{enumerate}
\usepackage{tikz-cd}
\usepackage[backref]{hyperref}
\usepackage{cleveref}
\usepackage{verbatim}
\usepackage{cancel}
\usepackage{float}
\usepackage{todonotes}
\usepackage{longtable}
\usepackage{listings}

\newtheorem{lem}{Lemma}[section]
\newtheorem{thm}[lem]{Theorem}

\newtheorem{prop}[lem]{Proposition}
\newtheorem{cor}[lem]{Corollary}

\theoremstyle{definition}
\newtheorem{remark}[lem]{Remark}
\newtheorem{definition}[lem]{Definition}

\DeclareMathAlphabet{\curly}{U}{rsfs}{m}{n}
\newcommand{\wideunder}{\rule{0.7em}{0.4pt}}

\newcommand{\End}{\operatorname{End}}

\newcommand{\Sym}{\operatorname{Sym}}
\newcommand{\Pic}{\operatorname{Pic}}

\newcommand{\Hom}{\operatorname{Hom}}

\newcommand{\Q}{\mathbb{Q}}
\newcommand{\C}{\mathbb{C}}

\newcommand{\Z}{\mathbb{Z}}
\newcommand{\N}{\mathbb{N}}

\newcommand{\GL}{\operatorname{GL}}

\newcommand{\SL}{\operatorname{SL}}

\DeclareMathOperator{\Id}{Id}

\newcommand{\PP}{{\mathbb P}}

\newcommand{\lmfdbec}[3]{\href{https://www.lmfdb.org/EllipticCurve/Q/#1/#2/#3}{#1.#2#3}}

\mathchardef\mhyphen="2D

\newcommand{\maarten}[1]{{\color{red} \sf $\spadesuit\spadesuit\spadesuit$ Maarten: [#1]}}

\title{Intermediate modular curves with infinitely many quartic points}

\author{\sc Maarten Derickx}
\address{Maarten Derickx\\
University of Zagreb\\  
Bijeni\v{c}ka Cesta 30 \\
10000 Zagreb\\
Croatia}
\email{maarten@mderickx.nl}
\urladdr{http://www.maartenderickx.nl/}

\author{\sc Petar Orli\'c}
\address{Petar Orli\'c \\
University of Zagreb\\  
Bijeni\v{c}ka Cesta 30 \\
10000 Zagreb\\
Croatia}
\email{petar.orlic@math.hr}

\begin{document}
\begin{abstract}
    For every group $\{\pm1\}\subseteq \Delta\subseteq (\Z/N\Z)^\times$, there exists an intermediate modular curve $X_\Delta(N)$. In this paper we determine all curves $X_\Delta(N)$ with infinitely many points of degree $4$ over $\Q$. To do that, we developed a method to compute possible degrees of rational morphisms from $X_\Delta(N)$ to an elliptic curve.
\end{abstract}


\maketitle

\section{Introduction}\label{introductionsection}

Let $C$ be a smooth curve defined over a number field $K$ and $d$ a positive integer. A point $P\in C$ is of degree $d$ over $K$ if $[K(P):K]=d$. An important property of a curve $C/K$ is whether it has finitely or infinitely many points of degree $d$. Faltings' Theorem solves the case $d=1$.

\begin{thm}[Faltings' Theorem]
Let $K$ be a number field and let $C$ be a non-singular curve defined over $K$ of genus $g\geq2$. Then the set $C(K)$ is finite.
\end{thm}

Therefore, the curve $C$ has infinitely many $K$-rational points if and only if $C$ is isomorphic to $\mathbb{P}^1$ ($g=0$) or $C$ is an elliptic curve ($g=1$) with positive $K$-rank. The next theorem solves the case $d=2$.

\begin{thm}[{Harris, Silverman: \cite[Corollary 3]{HarrisSilverman91}}]
    Let $K$ be a number field, and let $C/K$ be a curve of genus at least $2$. Assume that $C$ is neither hyperelliptic nor bielliptic. Then the set of quadratic points on $C/K$ is finite.
\end{thm}

\begin{definition}
    For a curve $C$ defined over a field $K$, the $K$-gonality $\textup{gon}_K C$ is the smallest integer $d$ such that there exists a morphism of degree $d$ from $C$ to $\mathbb{P}^1$ defined over $K$.
\end{definition}

The question of determining whether the curve $C/K$ has infinitely many degree $d$ points over $K$ is closely related to the $K$-gonality of $C$. For example, Frey \cite{frey} proved that if a curve $C$ defined over a number field $K$ has infinitely many points of degree $\leq d$ over $K$, then $\textup{gon}_K C\leq2d$.

The problem of determining degree $d$ points is particularly interesting in the case of modular curves. Modular curves $X(\Gamma)$ are a type of algebraic curves which can be constructed as quotients of the compactified upper half plane $\mathcal{H}^*$ with $\Gamma$, a congruence subgroup of the modular group $\textup{SL}_2(\Z)$. Examples of modular curves are $X(N)$, $X_0(N)$, $X_1(N)$ which correspond to congruence subgroups

\begin{align*}
    \Gamma(N)&=\left\{ 
\begin{bmatrix}
a & b\\
c & d
\end{bmatrix}
\in \SL_2(\Z) : a,d\equiv 1\pmod{N}, \ b,c\equiv 0\pmod{N}
\right\},\\
\Gamma_0(N)&=\left\{ 
\begin{bmatrix}
a & b\\
c & d
\end{bmatrix}
\in \SL_2(\Z) : c\equiv 0\pmod{N} \right\},\\
\Gamma_1(N)&=\left\{ 
\begin{bmatrix}
a & b\\
c & d
\end{bmatrix}
\in \SL_2(\Z) : a,d\equiv 1\pmod{N}, \ c\equiv 0\pmod{N}
\right\}.
\end{align*}

For every group $\Delta\subseteq (\Z/N\Z)^\times$, there exists a modular curve $X_\Delta(N)$ defined over $\Q$. It corresponds to the congruence subgroup 
$$\Gamma_\Delta(N)=\left\{ 
\begin{bmatrix}
a & b\\
c & d
\end{bmatrix}
\in \SL_2(\Z) : (a\textup{ mod } N)\in\Delta, \ c\equiv 0\pmod{N}
\right\}.$$

Since $-I\in\SL_2(\Z)$ acts trivially on the upper half plane $\mathcal{H}^*$, the curves $X_\Delta(N)$ and $X_{\pm\Delta}(N)$ are isomorphic. Therefore, in this paper we will always assume that $-1\in\Delta$.

For every group $\{\pm1\}\subseteq \Delta\subseteq (\Z/N\Z)^\times$, there exists a modular curve $X_\Delta(N)$. From the definition of $\Gamma_\Delta(N)$ it follows that \[X_{\{\pm1\}}(N)=X_{\pm1}(N)\cong X_1(N), \ X_{(\Z/N\Z)^\times}(N)=X_0(N).\] Moreover, if 
\[\{\pm1\}\subseteq \Delta_1\subseteq \Delta_2 \subseteq (\Z/N\Z)^\times,\]
then we have natural projections \[X_1(N)\to X_{\Delta_1}(N)\to X_{\Delta_2}(N)\to X_0(N)\] defined over $\Q$. If $\{\pm1\}\subsetneq \Delta\subsetneq (\Z/N\Z)^\times$, then we call the curve $X_\Delta(N)$ an intermediate modular curve (since it lies between the curves $X_1(N)$ and $X_0(N)$).

It is also possible to define modular curves in another, more general way. For every group $H\subseteq \GL_2(\Z/N\Z)$, there exists a modular curve $X_H$ defined over $\Q$. Since $-I\in\GL_2(\Z/N\Z)$ acts trivially on $\mathcal{H}^*$, we may assume that $-I\in H$ without loss of generality. If $H$ has full determinant (that is, if $\det H=(\Z/N\Z)^\times$), the curve $X_H$ is guaranteed to be geometrically irreducible.

Suppose that $H\subseteq \GL_2(\Z/N\Z)$ such that $-I\in H$ and that $H$ has full determinant. Then there is a congruence subgroup $\Gamma$ such that the curves $X_H$ and $X(\Gamma)$ are isomorphic. It is defined as follows:
$$H_0:=\SL_2(\Z/N\Z)\cap H, \ \Gamma:=\{A\in\SL_2(\Z) : (A \textup{ mod } N) \in H_0\}.$$
It is not hard to check that $\Gamma(N)\subseteq \Gamma$, therefore $\Gamma$ is indeed a congruence subgroup of $\SL_2(\Z)$. Conversely, in the case of intermediate modular curves $X_\Delta(N)$, its isomorphic curve $X_H$ is defined as 
$$H:=\left\{  
\begin{bmatrix}
a & b\\
c & d
\end{bmatrix}
\in \GL_2(\Z/N\Z) : a\in\Delta, \ c=0
\right\}.$$
We can easily see that $-I\in H$ and that $H$ has full determinant.

Low degree points over $\Q$ on modular curves have been extensively studied. We first present the results for the modular curve $X_1(M,N)$. All curves $X_1(M,N)$ with infinitely many points of degree $d\leq6$ were determined by Mazur \cite{mazur77} (for $d=1$), Kenku, Momose, and Kamienny \cite{KM88, kamienny92} (for $d=2$), Jeon, Kim, and Schweizer \cite{JeonKimSchweizer04} (for $d=3$), Jeon, Kim, and Park \cite{JeonKimPark06} (for $d=4$), and Derickx and Sutherland \cite{DerickxSutherland17} (for $d=5,6$). Additionally, Derickx and van Hoeij \cite{derickxVH} determined all curves $X_1(N)$ which have infinitely many points of degree $d=7,8$. 

These results for $X_1(M,N)$ are very important because they determine all possible torsion groups that appear infinitely often for elliptic curves defined over number fields of degree $d\leq6$. Moreover, all possible sporadic torsion groups (that is, those that appear for only finitely many elliptic cures) were determined for $d\leq4$. The paper \cite{Deg3Class} solves the case $d=3$ and the paper \cite{DerickxNajman24} solves the case $d=4$.

Now we consider the modular curve $X_0(N)$. All curves $X_0(N)$ with infinitely many degree $d$ points have been determined for $d\leq4$. Ogg \cite{Ogg74} determined all hyperelliptic curves $X_0(N)$, Bars \cite{Bars99} determined all bielliptic curves $X_0(N)$ and solved the case $d=2$, and Jeon \cite{Jeon2021} solved the case $d=3$. Hwang and Jeon \cite{Hwang2023}, and Derickx and Orlić \cite{DerickxOrlic23} independently solved the case $d=4$ and together \cite{DensityDegree5} made progress on the case $d=5$, leaving only finitely many unsolved levels $N$.

This paper focuses on intermediate curves $X_\Delta(N)$. All such curves with infinitely many degree $d$ points have been determined for $d\leq3$. These results are given below.

\begin{thm}[Jeon: {\cite[Theorem 0.3]{JEON2020272}}]\label{quadraticthm}
    Let $\{\pm1\}\subsetneq \Delta\subsetneq (\Z/N\Z)^\times$. The only modular curve $X_\Delta(N)$ of genus $g>1$ which has infinitely many quadratic points over $\Q$ is the unique hyperelliptic curve $X_{\Delta_1}(21)$, where $\Delta_1=\{\pm1,\pm8\}$.
\end{thm}

\begin{thm}[Dalal: {\cite[Theorem 1]{Dalal2023}}]\label{cubicthm}
    Let $\{\pm1\}\subsetneq \Delta\subsetneq (\Z/N\Z)^\times$. The modular curve $X_\Delta(N)$ has infinitely many cubic points over $\Q$ if and only if $g_{X_\Delta(N)}\leq1$ or $N$ and $\Delta$ are in the following list:

    \begin{center}
\begin{longtable}{|c|c|c|}

\hline
\addtocounter{table}{-1}
$N$ & $\Delta$ & $g_{X_\Delta(N)}$\\
    \hline

    $24$ & $\{\pm1,\pm5\}$ & $3$\\
    $24$ & $\{\pm1,\pm7\}$ & $3$\\
    $26$ & $\{\pm1,\pm5\}$ & $4$\\
    $26$ & $\{\pm1,\pm3,\pm9\}$ & $4$\\
    $28$ & $\{\pm1,\pm13\}$ & $4$\\
    $28$ & $\{\pm1,\pm3,\pm9\}$ & $4$\\
    $29$ & $\{\pm1,\pm4,\pm5,\pm6,\pm7,\pm9,\pm13\}$ & $4$\\
    $36$ & $\{\pm1,\pm11,\pm13\}$ & $3$\\
    $37$ & $\{\pm1,\pm6,\pm8,\pm10,\pm11,\pm14\}$ & $4$\\
    $37$ & $\{\pm1,\pm3,\pm4,\pm7,\pm9,\pm10,\pm11,\pm12,\pm16\}$ & $4$\\
    $49$ & $\{\pm1,\pm6,\pm8,\pm13,\pm15,\pm20,\pm22\}$ & $3$\\
    $50$ & $\{\pm1,\pm9,\pm11,\pm19,\pm21\}$ & $4$\\
 
    \hline

\caption*{}
\end{longtable}
\end{center}
\end{thm}

Furthermore, all curves $X_\Delta(N)$ with $\Q$-gonality equal to $d$ are known for $d=2$ \cite{IshiiMomose}, $d=3$ \cite{Jeon2007}, and $d=4,5$ \cite{Orlic2024Intermediate}. 

In this paper, continuing the presented results regarding intermediate modular curves, we determine all $X_\Delta(N)$ with infinitely many degree $4$ points. Our main result is the following theorem.

\begin{thm}\label{mainthm}
    Let $\{\pm1\}\subsetneq \Delta\subsetneq (\Z/N\Z)^\times$. The modular curve $X_\Delta(N)$ has infinitely many quartic points over $\Q$ if and only if \[N\in\{13,15,16,17,19,20,21,24,25,26,27,28,30,32,36\}\] or $N$ and $\Delta$ are in the following list. For larger groups $\Delta$ it is unfeasible to list all their elements. Therefore, we give the generators and the size of $\Delta$ in those cases. We also present the group structure and the generators of $(\Z/N\Z)^\times$.
    
\begin{center}
\begin{longtable}{|c|c|c|c|c|}

\hline
\addtocounter{table}{-1}
$N$ & $(\Z/N\Z)^\times$ & \textup{generators} & $\Delta$ & $g_{X_\Delta(N)}$\\
    \hline

    $29$ & $C_{28}$ & $2$ & $\{\pm1,\pm4,\pm5,\pm6,\pm7,\pm9,\pm13\}$ & $4$\\
    $33$ & $C_2\times C_{10}$ & $2,10$ & $\{\pm1,\pm2,\pm4,\pm8,\pm16\}$ & $5$\\
    $34$ & $C_{16}$ & $3$ & $\{\pm1,\pm9,\pm13,\pm15\}$ & $5$\\
    $35$ & $C_2\times C_{12}$ & $2,12$ & $\{\pm1,\pm4,\pm6,\pm9,\pm11,\pm16\}$ & $5$\\
     & & & $\{\pm1,\pm6,\pm8,\pm13\}$ & $7$\\
    $37$ & $C_{26}$ & $2$ & $\{\pm1,\pm3,\pm4,\pm7,\pm9,\pm10,\pm11,\pm12,\pm16\}$ & $4$\\
     & & & $\{\pm1,\pm6,\pm8,\pm10,\pm11,\pm14\}$ & $4$\\
    $39$ & $C_2\times C_{12}$ & $2,12$ & $\{\pm1,\pm4,\pm10,\pm14,\pm16,\pm17\}$ & $5$\\
     & & & $\{\pm1,\pm5,\pm8,\pm14\}$ & $9$\\
    $40$ & $C_2\times C_2\times C_4$ & $-1,5,7$ & $\{\pm1,\pm9,\pm11,\pm19\}$ & $5$\\
     & & & $\{\pm1,\pm7,\pm9,\pm17\}$ & $7$\\
     & & & $\{\pm1,\pm3,\pm9,\pm13\}$ & $7$\\
    $41$ & $C_{40}$ & $6$ & $\{\pm1,\pm2,\pm4,\pm5,\pm8,\pm9,\pm10,\pm16,\pm18,\pm20\}$ & $5$\\
    $45$ & $C_2\times C_{12}$ & $-1,2$ & $\{\pm1,\pm4,\pm11,\pm14,\pm16,\pm19\}$ & $5$\\
    $48$ & $C_2\times C_2\times C_4$ & $-1,5,7$ & $\{\pm1,\pm11,\pm13,\pm23\}$ & $5$\\
     & & & $\{\pm1,\pm7,\pm17,\pm23\}$ & $7$\\
     & & & $\{\pm1,\pm5,\pm19,\pm23\}$ & $7$\\
    $49$ & $C_{42}$ & $3$ & $\{\pm1,\pm6,\pm8,\pm13,\pm15,\pm20,\pm22\}$ & $3$\\
    $50$ & $C_{20}$ & $3$ & $\{\pm1,\pm9,\pm11,\pm19,\pm21\}$ & $4$\\
    $53$ & $C_{52}$ & $2$ & $\Delta=\left<4\right>$, $\#\Delta=26$ & $8$\\
    $55$ & $C_2\times C_{20}$ & $-1,2$ & $\Delta=\left<-1,4\right>$, $\#\Delta=20$ & $9$\\
    $61$ & $C_{60}$ & $2$ & $\Delta=\left<4\right>$, $\#\Delta=30$ & $8$\\
    $64$ & $C_{32}$ & $3$ & $\Delta=\left<9\right>$, $\#\Delta=16$ & $5$\\
    $65$ & $C_4\times C_{12}$ & $2,12$ & $\Delta=\left<2,12^2\right>$, $\#\Delta=24$ & $9$\\
     & & & $\Delta=\left<4,12\right>$, $\#\Delta=24$ & $11$\\
     & & & $\Delta=\left<4,24\right>$, $\#\Delta=24$ & $11$\\
    $75$ & $C_2\times C_{20}$ & $-1,2$ & $\Delta=\left<-1,4\right>$, $\#\Delta=20$ & $9$\\
    $89$ & $C_{88}$ & $3$ & $\Delta=\left<9\right>$, $\#\Delta=44$ & $9$\\
    $101$ & $C_{100}$ & $2$ & $\Delta=\left<4\right>$, $\#\Delta=50$ & $16$\\
 
    \hline
\caption*{}
\end{longtable}
\end{center}
\end{thm}

This paper is organized as follows. \begin{itemize}
    \item \Cref{preliminariessection} contains preliminary results that eliminate all but finitely many levels $N$ in \Cref{mainthm}.
    \item \Cref{infinitelyquarticsection} contains results which prove that all curves $X_\Delta(N)$ listed in \Cref{mainthm} have infinitely many quartic points.
    \item \Cref{tetraellipticsection} contains results which prove that all other curves $X_\Delta(N)$ do not admit a degree $4$ rational morphism to a positive rank elliptic curve $E$. To do that, we determine all possible degrees of rational morphisms to $E$.
    \item In \Cref{sectionmainthm} we combine these results to prove \Cref{mainthm}.
\end{itemize}

In \Cref{tetraellipticsection} we have two cases regarding the conductor of $E$: $\textup{Cond }E=N$ and $\textup{Cond E}<N$. In the case $\textup{Cond }E=N$ we determine in \Cref{deltastrongweildegree} the minimal degree of a parametrization $f': X_\Delta(N)\to E$. This is an important result because it is a generalization of the modular degree of elliptic curves over $\Q$ (the modular degree data is known and available on LMFDB). It is particularly interesting for curves $X_1(N)$.

\begin{prop}[{Special case of \Cref{deltastrongweildegree}}]
    Let $E/\Q$ be a strong Weil curve of conductor $N$. Suppose that $d$ is the minimal degree of a parametrization $f: X_0(N)\to E$. If $E'/\Q$ is a strong Weil curve for $X_1(N)$ in the isogeny class of $E$, then the minimal degree of a parametrization $f':X_1(N)\to E'$ equals \[\frac{d\cdot\varphi(N)}{\#(E\cap \Sigma(N))},\] where $\varphi$ is the Euler's totient function and $\Sigma(N)$ is the Shimura subgroup of $J_0(N)$ (defined in \Cref{section_condN}).
\end{prop}

Since in \cite{Orlic2024Intermediate} all intermediate curves $X_\Delta(N)$ with $\Q$-gonality equal to $5$ were also determined, along with those with $\Q$-gonality equal to $4$, one might ask if it is possible to determine all curves $X_\Delta(N)$ with infinitely many quintic points. However, we encounter several problems with the methods we use in this paper.

First, in all the cases when $X_\Delta(N)$ has infinitely many quartic points, we cannot use any of the results about the minimum density degree (defined in \Cref{densitydegreedef}), such as \Cref{kadetsvogt1.4} and other results in \cite{KadetsVogt}. There are results that characterize the curves with infinitely many degree $d$ points, for example:

\begin{thm}[{\cite[Theorem 4.2. (1)]{BELOV}}]\label{translate_of_abelian_variety_thm}
    Let $C$ be a curve over a number field $K$. We define \[W_d^r(C):=\{D\in\Pic^d(C): \ell(D)>r\},\] \[\phi_d:\Sym^d(C)\to\Pic^d (C), \ x_1+\ldots+x_d\mapsto[x_1+\ldots+x_d].\] Then $C$ has infinitely many degree $d$ points over $K$ if and only if at least one of the following two statements holds:
    \begin{enumerate}[(1)]
        \item There exists a map $C\to \mathbb{P}^1$ of degree $d$ over $K$.
        \item There exists a degree $d$ point $x\in C$ and a positive rank abelian subvariety $A\subset \Pic^0 (C)$ such that $\phi_d(x_1+\ldots+x_d)+A\subset W^0_dC$, where $x_1,\ldots,x_d$ are the Galois conjugates of $x$.
    \end{enumerate} 
\end{thm}

\begin{prop}[{\cite[Proposition 3.6]{DebarreFahlaoui}}]\label{DF_AV_dim}
    Let $C$ be a curve over $\C$ of genus $g$. Assume that $W_d^r(C)$ contains an abelian variety $A$ and that $d\leq g-1+r$. Then $\dim A\leq d/2-r$.
\end{prop}

\begin{cor}
    Let $C/\Q$ be a curve of genus $g\geq6$ and $\Q$-gonality $\geq6$. If $C$ has infinitely many quintic points, then the Jacobian $J(C)$ contains a positive rank abelian variety $A$ of dimension $\leq2$.
\end{cor}

However, these results will not be enough for some curves $X_\Delta(N)$. One of the problems that arise is that it is difficult to determine whether there exists a degree $5$ rational morphism to $\PP^1$ if the $\Q$-gonality is $\leq4$. See the more detailed discussion on that for curves $X_0(N)$ in \cite[Section 7.1]{DensityDegree5}, similar issues appear for curves $X_\Delta(N)$.

Even if we restrict ourselves to those curves that have only finitely many degree $4$ points (like in \cite{DensityDegree5}), we face difficulties. 
Although we can determine all positive rank pentaelliptic curves using \Cref{quadraticform_coefficients}, \Cref{degree5morphismcor} can only be applied to curves of genus $g\geq12$ and there are several curves $X_\Delta(N)$ of genus $6\leq g\leq11$ with finitely many quartic points such that the Jacobian $J_0(N)$ is of positive rank.

We might attempt to use \cite[Proposition 6.3]{KadetsVogt} instead since the genus bound there is $g\geq9$ instead of $g\geq12$, but a big bottleneck there is the characterization of Debarre-Fahlaoui curves, defined in \cite[Section 5]{KadetsVogt}. They are not defined intrinsically and we do not know how to prove in the general case that a curve is not Debarre-Fahlaoui. The only result in that direction we know is that Debarre-Fahlaoui curves must admit a non-constant morphism to an elliptic curve \cite[Proposition 4.6]{DensityDegree5}, but this is not enough to solve all problematic curves $X_\Delta(N)$.

The results in this paper are based on \texttt{Magma} \cite{magma} and Sage computations. The codes that verify all computations in this paper can be found on

\begin{center}
    \url{https://github.com/nt-lib/quartic-Xdelta}
\end{center}

\section*{Acknowledgments}

The authors were supported by the Croatian Science Foundation under the project no. IP-2022-10-5008.

We are grateful to Filip Najman for his comments on the paper.

\section{Preliminaries}\label{preliminariessection}

The first step to prove \Cref{mainthm} is to limit the number of possible levels $N$ such that the curve $X_\Delta(N)$ has infinitely many quartic points. Since there is a projection map $X_\Delta(N)\to X_0(N)$ defined over $\Q$, we conclude that in that case the curve $X_0(N)$ must also have infinitely many quartic points.

\begin{thm}[{\cite{Hwang2023, DerickxOrlic23}}]\label{quarticthmx0n}
    The modular curve $X_0(N)$ has infinitely many points of degree $4$ over $\Q$ if and only if
    \begin{align*}
        N\in\{&1-75,77-83,85-89,91,92,94-96,98-101,103,104,107,111,\\
        &118,119,121,123,125,128,131,141-143,145,155,159,167,191\}.
    \end{align*}
\end{thm}
Moreover, from the group structure of the group $(\Z/N\Z)^\times$, we can easily see that for
\[N\in\{1-12,14,18,22,23,46,47,59,83,94,107,167\}\]
there are actually no intermediate modular curves $X_\Delta(N)$ because in these cases $(\Z/N\Z)^\times\cong \Z/2p\Z$ for some prime $p$. This leaves us with reasonably many curves $X_\Delta(N)$ we need to study.

\begin{definition}\cite[Page 1]{KadetsVogt}\label{densitydegreedef}
    Let $C$ be a curve defined over a number field $K$. The {\it density degree set} $\delta(C/K)$ is the set of integers $d$ for which the collection of closed points of degree $d$ on $C$ is infinite. The {\it minimum density degree}\footnote{In previous versions of \cite{KadetsVogt}, the minimum density degree was called \textit{the arithmetic degree of irrationality} $\textup{a.irr}_K C$.} $\min(\delta(C/K))$ is the smallest integer $d\geq1$ such that $C$ has infinitely many closed points of degree $d$ over $K$, or equivalently:
    \[\min(\delta(C/K)) =\min\left\{d 
\in \N \mid \#\left\{\cup_{[F:K]=d} C(F)\right\}=\infty\right\}.\]
\end{definition}

In \Cref{infinitelyquarticsection} we will see that all curves $X_\Delta(N)$ with infinitely many points of degree $\leq3$ also have infinitely many quartic points. This means that for curves not listed in \Cref{mainthm} it will be enough to prove $\min(\delta(C/K))\neq4$. The following result is very useful for that. It tells us that for curves of large enough genus, infinitely many degree $d$ points must be preimages of smaller degree points.

\begin{thm}[\cite{KadetsVogt}, Theorem 1.4]\label{kadetsvogt1.4}
    Suppose $C/K$ is a curve of genus $g$ such that $\min(\delta(C/K))=d$. Let $m:=\lceil d/2\rceil -1$ and let $\epsilon:=3d-1-6m<6$. Then one of the following holds:
    \begin{enumerate}[(1)]
        \item There exists a nonconstant morphism of curves $\phi:C\to Y$ of degree at least $2$ such that $d=\min(\delta(Y/K))\cdot\textup{deg}\phi$.
        \item $g\leq\textup{max} \left(\frac{d(d-1)}{2}+1,3m(m-1)+m\epsilon\right)$.
    \end{enumerate}
\end{thm}

\begin{cor}[{\cite[Corollary 5.2]{DerickxOrlic23}}]\label{degree4morphismcor}
    Suppose $C/\Q$ is a curve of genus $g\geq8$ such that $\min(\delta(C/\Q))=4$. Then there exists a nonconstant morphism of degree $4$ from $C$ to $\mathbb{P}^1$ or an elliptic curve defined over $\Q$ with positive $\Q$-rank.
\end{cor}

\begin{cor}\label{degree5morphismcor}
    Suppose $C/\Q$ is a curve of genus $g\geq12$ such that $\min(\delta(C/\Q))=5$. Then there exists a nonconstant morphism of degree $5$ from $C$ to $\mathbb{P}^1$ or an elliptic curve defined over $\Q$ with positive $\Q$-rank.
\end{cor}

\section{Curves with infinitely many quartic points}\label{infinitelyquarticsection}

In this section, we will prove that all curves $X_\Delta(N)$ listed in \Cref{mainthm} have infinitely many quartic points. We will do this by finding a degree $4$ rational morphism from $X_\Delta(N)$ to $\PP^1$ or a positive rank elliptic curve. 
When $X_\Delta(N)$ has a map of degree $4$ to $\PP^1$, Hilbert's Irreducibility Theorem \cite[Proposition 3.3.5(2) and Theorem 3.4.1]{serre2016topics} implies the existence of infinitely many quartic points on $X_\Delta(N)$. 
In the cases where we found a degree $4$ map to a positive rank elliptic curve on $X_\Delta(N)$ there are no points of degree $<4$. This means that in this case the inverse images of rational points on the elliptic curve also have to contain infinitely many degree $4$ points.

\Cref{quadraticthm} gives us all curves $X_\Delta(N)$ with infinitely many quadratic points. Such curves automatically have infinitely many quartic points. The curves $X_\Delta(N)$ of genus $g\leq1$ are those for $N\in\{13,15,16,17,19,20\}$ and
\begin{align*}
   (N,\Delta)\in\{&(21,\{\pm1,\pm4,\pm5\}),(24,\{\pm1\pm11\}),(\{25,\{\pm1,\pm4,\pm6,\pm9,\pm11\},\\
   &(27,\{\pm1,\pm8,\pm10\})\}),(32,\{\pm1,\pm7,\pm9,\pm15\})\}.
\end{align*}

We now consider the curves $X_\Delta(N)$ of genus $g\geq2$. The only such hyperelliptic curve is $X_{\{\pm1,\pm8\}}(21)$. All other curves have only finitely many quadratic points.

\begin{prop}\label{gon4prop}
    The modular curve $X_\Delta(N)$ has infinitely many quartic points for $N$ and $\Delta$ in the following table:

\begin{center}
\begin{longtable}{|c|c|c|c|c||c|c|c|c|c|}

\hline
\addtocounter{table}{-1}
    $N$ & $\Delta$ & $N$ & $\Delta$\\
    \hline
    $25$ & $\{\pm1,\pm7\}$ &
    $30$ & $\{\pm1,\pm11\}$\\
    $32$ & $\{\pm1,\pm15\}$ &
    $33$ & $\{\pm1,\pm2,\pm4,\pm8,\pm16\}$\\
    $34$ & $\{\pm1,\pm9,\pm13,\pm15\}$ &
    $35$ & $\{\pm1,\pm6,\pm8,\pm13\}$\\
    $35$ & $\left<-1,4\right>$ &
    $36$ & $\{\pm1,\pm17\}$\\
    $39$ & $\{\pm1,\pm5,\pm8,\pm14\}$ &
    $39$ & $\left<-1,4\right>$\\
    $40$ & $\{\pm1,\pm3,\pm9,\pm13\}$ &
    $40$ & $\{\pm1,\pm7,\pm9,\pm17\}$\\
    $40$ & $\{\pm1,\pm9,\pm11,\pm19\}$ &
    $41$ & $\left<-1,2\right>$\\
    $45$ & $\left<-1,4\right>$ &
    $48$ & $\{\pm1,\pm5,\pm19,\pm23\}$\\
    $48$ & $\{\pm1,\pm7,\pm17,\pm23\}$ &
    $48$ & $\{\pm1,\pm11,\pm13,\pm23\}$\\
    $55$ & $\left<-1,4\right>$ &
    $64$ & $\left<-1,9\right>$\\
    $75$ & $\left<-1,4\right>$ & &\\
     
    \hline

\caption*{}
\end{longtable}
\end{center}
\end{prop}

\begin{proof}
    These curves have $\Q$-gonality equal to $4$ by \cite[Theorem 1.3]{Orlic2024Intermediate}.
\end{proof}

\begin{prop}\label{deg4mapprop}
    The modular curve $X_\Delta(N)$ has infinitely many quartic points for $N$ and $\Delta$ in the following table:
\end{prop}

\begin{center}
\begin{longtable}{|c|c|c|}

\hline
\addtocounter{table}{-1}
    $N$ & $\Delta$ & $Y$\\
    \hline
    $24$ & $\{\pm1,\pm5\}$ & $X_0(24)$\\
    $24$ & $\{\pm1,\pm7\}$ & $X_0(24)$\\
    $26$ & $\{\pm1,\pm3,\pm9\}$ & $X_0(26)$\\
    $28$ & $\{\pm1,\pm3,\pm9\}$ & $X_0(28)$\\
    $29$ & $\{\pm1,\pm4,\pm5,\pm6,\pm7,\pm9,\pm13\}$ & $X_0(29)$\\
    $36$ & $\{\pm1,\pm11,\pm13\}$ & $X_0(36)$\\
    $37$ & $\{\pm1,\pm3,\pm4,\pm7,\pm9,\pm10,\pm11,\pm12,\pm16\}$ & $X_0(37)$\\
    $50$ & $\{\pm1,\pm9,\pm11,\pm19,\pm21\}$ & $X_0(50)$\\
    $53$ & $\left<4\right>$ & $X_0(53)$\\
    $55$ & $\left<-1,4\right>$ & $X_0(55)$\\
    $61$ & $\left<4\right>$ & $X_0(61)$\\
    $65$ & $\left<2,12^2\right>$ & $X_0(65)$\\
    $65$ & $\left<4,12\right>$ & $X_0(65)$\\
    $65$ & $\left<4,24\right>$ & $X_0(65)$\\
    $89$ & $\left<9\right>$ & $X_0(89)$\\
    $101$ & $\left<4\right>$ & $X_0(101)$\\
     
    \hline

\caption*{}
\end{longtable}
\end{center}

\begin{proof}
    In all these cases we have $[\Gamma_0(N):\Delta]=2$, meaning that the projection map $X_\Delta(N)\to X_0(N)$ is of degree $2$. The curves $X_0(N)$ with $N \leq 50$ in this list have $\Q$-gonality equal to $2$, giving rise to a degree $4$ map to $\PP^1$ so Hilbert's irreducibility Theorem ensures the existence of infinitely many quartic points. In the cases where $N > 50$ the gonality of $X_0(N)$ is $>2$ but these curves still have infinitely many quadratic points by \cite[Theorem 4.3]{Bars99}. 
    The pullbacks of these quadratic points are a source of infinitely many quartic points on $X_\Delta(N)$ as it has only finitely many points of degree $\leq3$ by \Cref{quadraticthm} and \Cref{cubicthm}.
\end{proof}

\begin{prop}\label{deg4canonicaldivisor}
    The modular curve $X_\Delta(N)$ has infinitely many quartic points for $(N,\Delta)=(49,\left<27\right>)$.
\end{prop}

\begin{proof}
    This curve is of genus $3$ and has $\Q$-gonality equal to $3$. Since it is defined over $\Q$, there exists a canonical divisor defined over $\Q$. Let us take such a divisor $K=P_1+P_2+P_3+P_4$. Then we have $\deg K=2g-2=4$ and $\ell(K)=g=3$. Moreover, by applying the Riemann-Roch theorem for the divisor $Q=P_1+P_2+P_3$ we obtain \[\ell(Q)-\ell(K-Q)=3-3+1=1\implies \ell(Q)=\ell(P_4)+1=1+1=2.\] Since $K$ is defined over $\Q$, by \cite[Appendix B.12]{Stichtenoth09} we have that \[L_\Q(K)=\Q(X_\Delta(N))\cap L(D)\] is of the same dimension $\ell(K)=3$ over $\Q$. Furthermore, because $\ell(Q)=2$, there exists a rational function in $\Q(X_\Delta(N))$ whose polar divisor is equal to $K$ and this gives us the desired degree $4$ rational morphism to $\PP^1$.
\end{proof}

\begin{prop}\label{deg4lmfdb}
    The modular curve $X_\Delta(N)$ has infinitely many quartic points for $(N,\Delta)=(28,\{\pm1,\pm11\})$.
\end{prop}

\begin{proof}
    This curve is of genus $4$ and has $\Q$-gonality equal to $3$. Its LMFDB label is $28.144.4.d.1$. We can see this because $[\SL_2(\Z):\Delta]=144$ and it is minimally covered by a degree $2$ map from the curve $X_{\{\pm1\}}(28)$. This can be seen on \begin{center}
        \url{https://beta.lmfdb.org/ModularCurve/Q/28.288.10.g.1/}.
    \end{center}
    LMFDB also tells us that this curve $X_\Delta(N)$ is a minimal cover of $X_{\{\pm1\}}(14)$ of degree $2$. This is a rational morphism and the curve $X_{\{\pm1\}}(14)$ is an elliptic curve. This can be seen on \begin{center}
        \url{https://beta.lmfdb.org/ModularCurve/Q/28.144.4.d.1/}.
    \end{center}
    The pullbacks of quadratic points on $X_{\{\pm1\}}(14)$ are a source of infinitely many quartic points on $X_\Delta(N)$.
\end{proof}

\begin{prop}\label{deg4magma}
    The modular curve $X_\Delta(N)$ has infinitely many quartic points for \[(N,\Delta)\in\{(26,\{\pm1,\pm5\}),(37,\{\pm1,\pm6,\pm10,\pm10,\pm11,\pm14\})\}.\]
\end{prop}

\begin{proof}
    These curves are of genus $4$ and have $\Q$-gonality equal to $3$. Their LMFDB labels are $26.126.4.a.1$ and $37.114.4.b.1$. This can be seen by looking at the generators of these two curves on \begin{center}
        \url{https://beta.lmfdb.org/ModularCurve/Q/26.126.4.a.1/},
        \url{https://beta.lmfdb.org/ModularCurve/Q/37.114.4.b.1/}.
    \end{center}
    The canonical models of these two curves consist of the intersection of a quadric and a cubic. We retrieved them from LMFDB and used \texttt{Magma} to find a degree $4$ rational morphism to $\PP^1$.

    We did this by finding a degree $4$ effective rational divisor on $X_\Delta(N)$ of the form $1+1+1+1$ (that is, a sum of $4$ rational points) with Riemann-Roch dimension equal to $2$. Since such divisors are defined over $\Q$, \texttt{Magma} also computes their Riemann-Roch spaces over $\Q$ \cite{magma}, ensuring that the functions found in these Riemann-Roch spaces are also defined over $\Q$ \cite[Proof of Proposition 6.6]{DensityDegree5}.
\end{proof}

\section{Positive rank tetraelliptic curves}\label{tetraellipticsection}

We will call the curve $C/\Q$ positive rank tetraelliptic if it admits a degree $4$ morphism to a positive rank elliptic curve $E/\Q$.

All intermediate curves $X_\Delta(N)$ for $N$ in \Cref{quarticthmx0n} not listed in \Cref{mainthm} have genus $g\geq8$. Furthermore, \cite[Theorem 1.3]{Orlic2024Intermediate} tells us that all these curves have $\Q$-gonality greater than $4$. From \Cref{degree4morphismcor} it follows that it is enough to prove that these curves $X_\Delta(N)$ do not admit a degree $4$ rational morphism to a positive rank elliptic curve.

Suppose that there is a degree $4$ rational morphism $f:X_\Delta(N)\to E$, where $E/\Q$ is of positive rank. Since all curves $X_\Delta(N)$ have rational cusps, we can assume that $f(P)=0$ for some fixed $P\in X_\Delta(N)(\Q)$ (because $f$ and $f-f(P)$ have the same degree). Due to the universal property of the Jacobian \cite[Theorem 6.1]{Milne1986}, $f$ factors through $J_\Delta(N)$. Therefore, $E$ appears as an isogeny factor in $J_\Delta(N)$.

The space of modular forms $S_2(\Gamma_1(N))$ decomposes into subspaces according to the nebentypus character. By definition, $S_2(\Gamma_0(N)) \subseteq S_2(\Gamma_1(N))$ is the subspace consisting of modular forms with trivial character \cite[Section 4.3]{modular}. This implies that the simple isogeny factors of $J_0(N)$ are exactly the simple factors of $J_\Delta(N)$ whose associated modular form has trivial nebentypus character. Modular forms associated with elliptic curves $E/\Q$ have trivial character \cite[Section 6.6]{modular}. In particular, this means that an elliptic curve $E$ appears as an isogeny factor in $J_0(N)$ with the same multiplicity as in $J_\Delta(N)$. The conductor of an elliptic curve $E$ is the same as the level of the associated newform, so it follows that $\textup{Cond}(E) \mid N$ for all $1$-dimensional isogeny factors of $J_\Delta(N)$.

\begin{prop}\label{deltatetraelliptic_rank0}
    The modular curve $X_\Delta(N)$ is not positive rank tetraelliptic for
    \begin{align*}
        N\in\{&29,31,33,34,35,38,39,40,41,42,44,45,48,49,50,51,52,54,55,56,60,62,63,64,\\
        &66,68,69,70,71,72,75,78,80,81,85,87,95,96,98,100,103,104,109,125,191\}.
    \end{align*} and all possible $\Delta$.
\end{prop}

\begin{proof}
    For all these levels $N$ there are no positive rank elliptic curves $E/\Q$ with conductor dividing $N$. For $N\neq85,103,125,191$ this is because the Jacobian $J_0(N)$ is of rank $0$ over $\Q$ by \cite[Theorem 3.1]{Deg3Class} and for $N\in\{85,103,125,191\}$ the Jacobian $J_0(N)$ is of positive rank, but its decomposition contains no elliptic curves of positive rank. 
\end{proof}

It remains to study the situation when the decomposition of the Jacobian $J_\Delta(N)$ contains positive rank elliptic curves $E/\Q$. We divide the discussion into two cases: $\textup{Cond }E=N$ and $\textup{Cond }E<N$.

\subsection{Degree pairing and covers of curves.}\label{degreepairing-section}
In \cite[\S 2]{DerickxOrlic23} there is an outline of a strategy on how to determine whether a given curve $C$ over $\Q$ of genus $>1$ has a map of degree $d$ to an elliptic curve. The key result in this strategy is the following proposition:

\begin{prop}[{\cite[Proposition 1.10]{DerickxOrlic23}}]\label{prop:derickx-orlic-quadratic-form}
Let $C$ be a curve over $\mathbb Q$ with at least one rational point and $E$ an elliptic curve over $\mathbb Q$ that occurs as an isogeny factor of $J(C)$ with multiplicity $n \geq 1$. Then the degree map $\deg: \Hom_\Q(C,E) \to \Z$ can be extended to a positive definite quadratic form on $\Hom_\Q(J(C),E) \cong \Z^n$.
\end{prop}

Using the above proposition, the problem of deciding whether there exists a degree $d$ map $C \to E$ reduces to checking whether $d$ is a value of the quadratic form in \Cref{prop:derickx-orlic-quadratic-form}.

The following definition gives a more precise description of this quadratic form.

\begin{definition}[{\cite[Definition 2.2]{DerickxOrlic23}}]
Let $C,E$ be curves over a field $k$ with $E$ being an elliptic curve. The degree pairing is defined on $\Hom(C,E)$ as
\begin{align*}
\left<\wideunder,\wideunder\right> : \Hom(C,E)\times\Hom(C,E) & \to \End(J(E))\\
f,g &\mapsto f_*\circ g^*.
\end{align*}
If $P \in C(k)$, then we can use the map $f_P:C\to J(C)$, $x\mapsto[x-P]$ to define the degree pairing on $\Hom(J(C),E)$ as
\begin{align*}
\left<\wideunder,\wideunder\right> : \Hom(J(C),E)\times\Hom(J(C),E) & \to  \End(J(E)),\\
f,g &\mapsto (f \circ f_P)_*\circ (g \circ f_P)^*.
\end{align*}

\end{definition}

\cite[Proposition 2.5]{DerickxOrlic23} tells us that the degree pairing on $\Hom(J(C),E)$ is bilinear. We will also write $\left<f,g\right>:=f_*\circ g^*$ for $f,g\in\Hom(C,C')$ (this is not a pairing when $C'$ is not elliptic since $\Hom(C,C')$ is not an abelian group in that case). With this notation we have $\langle f, f \rangle = [\deg f]$ for $f \in \Hom(C, C')$ \cite[Lemma 2.1]{DerickxOrlic23}.

As we have already seen in the introduction of this section, any elliptic curve $E$ over $\Q$ occurs with the same multiplicity in the isogeny decompositions of $J_\Delta(N)$ and $J_0(N)$. In this subsection we will study more generally how the quadratic form from \Cref{prop:derickx-orlic-quadratic-form} changes when $u: C \to C'$ is a map of curves over $\Q$ and $E$ occurs with the same multiplicity in $J(C)$ and $J(C')$.

The following lemmas encapsulate the key idea we will use to determine possible degrees of morphisms $X_\Delta(N)\to E$.

\begin{lem}\label{lem:degree_pairing_covering_curve}
Let $u: C\to C'$ be a map of nice curves (smooth, projective, and geometrically integral) over $\Q$, and $E/\Q$ an elliptic curve that appears with multiplicity $n>0$ in both $J(C)$ and $J(C')$. 
Let $f_1,\ldots,f_n: C'\to E$ be maps such that $f_{1,*},\ldots,f_{n,*}$ form a basis for $\Hom_\Q(J(C'),E)\cong \Z^n$. Let $A \subseteq J(C')$ be the unique abelian subvariety isogenous to $E^n$ and assume that $C(\Q) \neq \emptyset$ and that $u^*_{\restriction_A} : A \to J(C)$ has trivial kernel. Then we have the following:
\begin{enumerate}
    \item $(f_{1} \circ u)_*,\ldots,(f_{1} \circ u)_*$ form a basis for $\Hom_\Q(J(C), E)$.
    \item The degree pairing on $\Hom_\Q(J(C), E)$ is the degree pairing on $\Hom_\Q(J(C'), E)$ multiplied by $\deg u$.
\end{enumerate}
\end{lem}

\begin{proof}
For the first part, let $\iota : A \to J(C')$ denote the inclusion map and $\iota^\vee :J(C')=J(C')^\vee \to A^\vee$ its dual. Then by assumption $u^*_{\restriction_A} = u^* \circ \iota : A \to J(C)$ has trivial kernel and hence $u^* \circ \iota$ induces an isomorphism between $A$ and its image $J(C)$. Since $E$ occurs with multiplicity $n$ in $J(C)$, the image of $A$ in $J(C)$ is actually the unique abelian subvariety of $J(C)$ isogenous to $E^n$. In particular, any map $f:E \to J(C)$ factors as $u^* \circ \iota$. Dually, any map $f:J(C) \to E$ can be written as $f' \circ (u^* \circ \iota)^\vee = f' \circ\iota^\vee \circ u_*$ for some $f' : A^\vee \to E$. Writing $f' \circ\iota^\vee = \sum a_i f_{i,*}$ shows that $f = (\sum a_i f_{i,*}) \circ u_*=\sum a_i (f_i\circ u)_*$ from which the first part follows.

The second part follows since if $i,j$ are integers between $1$ and $n$, then 
\begin{align*}
\langle f_{i} \circ u,f_{j} \circ u \rangle :=& (f_{i} \circ u)_* \circ (f_{i} \circ u)^* = f_{i,*} \circ u_* \circ u^* \circ f_j^* = \ldots \\
=& f_{i,*} \circ [\deg u] \circ f_{j}^* = [\deg u]\circ f_{i,*} \circ f_{j}^* = [\deg u] \langle f_i, f_j \rangle \end{align*}
\end{proof}

\begin{lem}\label{lem:degree_pairing_isogeny}
Let $C$ be a nice curve over $\Q$ with $C(\Q) \neq \emptyset$. Let $E, E'$ be isogenous elliptic curves over $\Q$ and $\phi: E \to E'$ an isogeny of minimal degree. Let $n$ be the multiplicity with which $E$ occurs in $J(C)$ and let $f_1,\ldots,f_n: C \to E$ be maps and define $\xi:= (f_{1,*},\ldots,f_{n,*}): J(C) \to E^n$. If $\xi^\vee : E^n \to J(C)$ has trivial kernel, then
\begin{enumerate}
    \item $f_{1,*},\ldots,f_{n,*}$ and $(\phi \circ f_{1})_*,\ldots,(\phi \circ f_{n})_*$ form a basis of $\Hom_\Q(J(C),E)$ and $\Hom_\Q(J(C),E')$, respectively.
    \item The degree pairing on $\Hom_\Q(J(C), E')$ is the degree pairing on $\Hom_\Q(J(C), E)$ multiplied by $\deg \phi$.
\end{enumerate}
\end{lem}
\begin{proof}
For the first part, let $f \in \Hom_\Q(J(C),E')$. Since $\xi^\vee$ is injective and the multiplicity of $E$ as an isogeny factor in $J(C)$ is equal to $n$, we know that $f^\vee : E' \to J(C)$ has to factor through $\xi^\vee$. Dually, this means that there is an $f' : E^n \to E'$ such that $f=f'\circ \xi$. There is a canonical isomorphism $\Hom_\Q(E^n, E')\cong\Hom_\Q(E, E')^n$ sending $f' \in \Hom_\Q(E^n, E')$ to its restriction on each factor of $E^n$. Now $\Hom_\Q(E, E')\cong \Z$ and is generated as a $\Z$-module by $\phi$. In particular, the restriction of $f'$ to each factor of $E^n$ factors through $\phi$, implying that $f'$ has to factor as $\phi \circ f''$ for some $f'' \in \Hom(E^n,E)$. Under the isomorphism $\Hom_\Q(E^n,E) \cong \Hom_\Q(E,E)^n=\Z^n$ we have $f''=([a_1],\ldots,[a_n])$ for some integers $a_1,\ldots, a_n$. Putting this all together means $$f=f'\circ \xi = \phi  \circ f''\circ \xi=\sum_{i=1}^n [a_i](\phi \circ f_i)_{*}.$$ This proves that $(\phi \circ f_{1})_*,\ldots,(\phi \circ f_{n})_*$ form a basis of $\Hom_\Q(J(C),E')$. To get that $f_{1,*},\ldots,f_{n,*}$ form a basis of $\Hom_\Q(J(C),E)$ just take $E=E'$ and $\phi = \Id_E$.

For the second part, we compute the degree pairing on the given basis. To this end, let $i,j$ be integers between $1$ and $n$. The second part follows from the computation \[\langle \phi \circ f_{i},\phi \circ f_{j}\rangle:=\phi_* \circ f_{i,*} \circ f_{j}^* \circ \phi^* = \phi_* \circ \langle f_i, f_j \rangle\circ \phi^* =\phi_*\circ \phi^* \circ \langle f_i, f_j \rangle =[\deg \phi] \langle f_i, f_j \rangle,\] where in the above we used that $\langle f_i, f_j \rangle \in \End_\Q(E)$ is multiplication by some integer so it commutes with $\phi^*$.
\end{proof}

\subsection{Conductor $N$}\label{section_condN}

Let $E/\Q$ be a strong Weil curve of conductor $N$. By \cite[Remark 3.4]{DerickxOrlic23}, the map $E\to J_0(N)$ induced by the minimal degree modular parametrization $f: X_0(N)\to E$ is injective. We also consider the map \[u^*: J_0(N)\to J_1(N)\] induced by the natural projection $u:X_1(N)\to X_0(N)$. Its kernel $\Sigma(N):=\ker u^*$ is called the Shimura subgroup. The kernel of the composition map \[E\to J_0(N)\xrightarrow{u^*} J_1(N)\] is $\Sigma(N)\cap E$ and is called the Shimura subgroup of $E$. If $\{\pm1\}\subseteq \Delta \subseteq(\Z/N\Z)^\times$ is a subgroup, then similarly we get a map $u^*_\Delta: J_0(N)\to J_\Delta(N)$. We denote its kernel by $\Sigma(\Delta)$. Since $u^*$ factors through $u^*_\Delta$, we have the inclusion $\Sigma(\Delta)\subseteq \Sigma(N)$.

\begin{prop}\label{shimurakernel2group}
    Let $E/\Q$ be a strong Weil curve of conductor $N$ and odd analytic rank. Then the kernel of the composition map \[E\to J_0(N)\to J_\Delta(N)\] is contained in $E[2]$ for all groups $\{\pm1\}\subseteq \Delta\subseteq (\Z/N\Z)^\times$.
\end{prop}

\begin{proof}
    The kernel of the composition is $E \cap \Sigma(\Delta)$, which is contained in $E \cap \Sigma(N)$. So, it suffices to show $E \cap \Sigma(N) \subset E[2]$.  By \cite[Chapter II, Proposition 11.7]{mazur77}, the Atkin-Lehner involution $w_N$ acts as $-1$ on $\Sigma(N)$. Furthermore, since $E$ has an odd analytic rank, it follows by looking at the functional equation for $L(E,s)$ that $w_N$ acts as $1$ on $E$. Therefore, $-1=1$ on $\Sigma(N)\cap E$ meaning that $\Sigma(N)\cap E\subseteq E[2]$.
\end{proof}

\begin{prop}\label{shimurakernelmagma}
 Let $E/\Q$ be a strong Weil curve of conductor $N\leq 800$ and odd analytic rank. Then the kernel of the composition map \[E\to J_0(N)\to J_\Delta(N)\] is trivial for all groups $\{\pm1\}\subseteq \Delta\subseteq (\Z/N\Z)^\times$.
\end{prop}

\begin{proof}
    As in \Cref{shimurakernel2group}, it suffices to show that the Shimura subgroup $\Sigma(N)\cap E$ is trivial. We computed this group using the SageMath function \texttt{E.shimura\_subgroup()}. The code is available on \begin{center}
        \url{https://github.com/nt-lib/quartic-Xdelta/blob/main/odd_rank.sage}
    \end{center}
\end{proof}

\begin{remark}
The observed triviality of the Shimura subgroup $\Sigma(N)\cap E$ is specific to elliptic curves with odd analytic rank. We also computed the Shimura subgroups of all elliptic curves of even analytic rank and conductor $N\leq 800$. Here we list all the cases where it is nontrivial. We give the level $N$ and the group structure of the Shimura subgroup.

    \begin{center}
\begin{longtable}{|c|c||c|c||c|c||c|c||c|c|}

\hline
\addtocounter{table}{-1}
    $N$ & $\Sigma(N)\cap E$ & $N$ & $\Sigma(N)\cap E$ & $N$ & $\Sigma(N)\cap E$ & $N$ & $\Sigma(N)\cap E$ & $N$ & $\Sigma(N)\cap E$\\
    \hline
$11$ & $C_5$ &
$14$ & $C_3$ &
$15$ & $C_4$ &
$17$ & $C_4$ &
$19$ & $C_3$\\
$20$ & $C_2$ &
$21$ & $C_2$ &
$24$ & $C_2$ &
$26$ & $C_3$ &
$27$ & $C_3$\\
$32$ & $C_2$ &
$33$ & $C_2$ &
$35$ & $C_3$ &
$37$ & $C_3$ &
$38$ & $C_3$\\
$39$ & $C_2$ &
$40$ & $C_2$ &
$48$ & $C_2$ &
$52$ & $C_2$ &
$54$ & $C_3$\\
$55$ & $C_2$ &
$57$ & $C_2$ &
$64$ & $C_2$ &
$73$ & $C_2$ &
$77$ & $C_3$\\
$80$ & $C_2$ &
$80$ & $C_2$ &
$89$ & $C_2$ &
$113$ & $C_2$ &
$116$ & $C_2$\\
$128$ & $C_2$ &
$128$ & $C_2$ &
$129$ & $C_2$ &
$158$ & $C_3$ &
$161$ & $C_2$\\
$208$ & $C_2$ &
$212$ & $C_2$ &
$233$ & $C_2$ &
$278$ & $C_3$ &
$291$ & $C_2$\\
$326$ & $C_3$ &
$353$ & $C_2$ &
$370$ & $C_3$ &
$464$ & $C_2$ &
$485$ & $C_3$\\
$593$ & $C_2$ &
$681$ & $C_2$ &
$692$ & $C_2$ & & & &\\
     
    \hline

\caption*{}
\end{longtable}
\end{center}
We see that the Shimura subgroup is cyclic for all elliptic curves $E/\Q$ of conductor $N\leq800$ and all cases when its size is $\geq4$ are those when $g(X_0(N))=1$. So, it makes sense to ask what is the possible size of $\Sigma(N)\cap E$ and is it bounded by $3$ when $N\geq18$. We can show that it is bounded. Indeed, the Cartier dual of the Shimura subgroup has trivial Galois action and hence consists entirely of rational points. In particular, Mazur's torsion theorem implies that $\#(\Sigma(N)\cap E) \leq 16$.
\end{remark}

We will use these two results in the next proposition to compute the degree of a parametrization $f': X_\Delta(N)\to E'$, where $E'/\Q$ is a strong Weil curve for $X_\Delta(N)$ in the isogeny class of $E$. Note that $E'$ may not be equal to $E$. For example, $E=X_0(11)$ is obviously a strong Weil curve for $X_0(11)$ and $E'=X_1(11)$ is obviously a strong Weil curve for $X_1(11)$. LMFDB labels of these curves are $E=\lmfdbec{11}{a}{2}$, $E'=\lmfdbec{11}{a}{3}$.
\begin{prop}\label{deltastrongweildegree}
    Let $E/\Q$ be a strong Weil curve of conductor $N$ and let us take $\{\pm1\}\subseteq \Delta\subseteq (\Z/N\Z)^\times$. Suppose that $d$ is the minimal degree of a parametrization $f: X_0(N)\to E$. If $E'/\Q$ is a strong Weil curve for $X_\Delta(N)$ in the isogeny class of $E$, then the minimal degree of a parametrization $f':X_\Delta(N)\to E'$ equals \[\frac{d\cdot \#((\Z/N\Z)^\times/\Delta)}{\#(E\cap \Sigma(\Delta))}.\]
\end{prop}

\begin{proof} We consider the composition map $g:X_\Delta(N)\to X_0(N)\to E$ of degree $d\cdot\#((\Z/N\Z)^\times/\Delta)$. Let us take a fixed point $P\in X_\Delta(N)$. We may (by translating) without loss of generality suppose that $g(P)=0_E$, $f'(P)=0_{E'}$. By the universal property of the Jacobian, $f$ and $g$ factor uniquely through $J_\Delta(N)$. Since $E'$ is a strong Weil curve for $X_\Delta(N)$, $g$ factors through $f'$ and we have \[g^*:E\xrightarrow{\varphi^*} E'\xrightarrow{(f')^*} J_\Delta(N)\] for some nonzero isogeny $\varphi:E'\to E$. Here we slightly abused the notation by writing $E$ and $E'$ instead of their Jacobians. Therefore, we have the following commutative diagram ($u_\Delta:X_\Delta(N)\to X_0(N)$ is the natural projection map):

\[ \begin{tikzcd}
X_\Delta(N) \arrow{r}{f'} \arrow[swap]{d}{u_\Delta} \arrow[dashrightarrow]{dr}{g} & E' \arrow{d}{\varphi} \\%
X_0(N) \arrow{r}{f}& E
\end{tikzcd}
\]

    The kernel of $g^*$ is $E \cap \Sigma(\Delta)$. Since $(f')^*$ is injective by \cite[Remark 3.4]{DerickxOrlic23}, we get that the kernel of $\varphi^*$ equals $E\cap\Sigma(\Delta)$ as well. This means that we have \[\deg f'\cdot \#(E \cap \Sigma(\Delta))=\deg f'\cdot \deg \varphi=\deg g=d\cdot\#((\Z/N\Z)^\times/\Delta).\] 
\end{proof}

Combining this with \Cref{shimurakernel2group} and \Cref{shimurakernelmagma} gives the following corollary.
\begin{cor}\label{deltatetraelliptic_condN}
    If $E$ has odd analytic rank, then the minimal degree of the modular parametrization $f':X_\Delta(N)\to E'$ is equal to the numerator of $d\cdot\#((\Z/N\Z)^\times/\Delta)/4$. If we additionally have $N\leq 800$, then it equals $d\cdot\#((\Z/N\Z)^\times/\Delta)$ and the composition $X_\Delta(N) \to X_0(N) \to E$ is a strong Weil parametrization.
\end{cor}

\begin{remark}
    From the proof of \Cref{deltastrongweildegree} it follows that for elliptic curves of odd analytic rank either $E\cong E'$ or they are 2-isogenous. Additionally, for $N\leq800$ we must have $E \cong E'$.

    This does not necessarily hold for elliptic curves with even analytic rank. For example, the previously mentioned curves $E=X_0(11)=\lmfdbec{11}{a}{2}$ and $E'=X_1(11)=\lmfdbec{11}{a}{3}$ have rank $0$ over $\Q$ and are $5$-isogenous but not $2$-isogenous.
\end{remark}

\Cref{deltastrongweildegree} is useful because the elliptic curve modular degree data is available on LMFDB. We will use this for many curves $X_\Delta(N)$ to prove that they are not positive rank tetraelliptic.

\begin{cor}\label{tetraelliptic_condN}
    The intermediate modular curve $X_\Delta(N)$ is not positive rank tetraelliptic for \[N\in\{43,57,58,67,73,77,79,82,88.91,92,99,121,123,128,131,141,143,145,155\}\] with all possible $\{\pm1\}\subseteq \Delta\subsetneq (\Z/N\Z)^\times$ and also for the following pairs of $(N,\Delta)$:
\begin{center}
\begin{longtable}{|c|c||c|c||c|c|}

\hline
\addtocounter{table}{-1}
    $N$ & $\Delta\neq$ & $N$ & $\Delta\neq$ & $N$ & $\Delta\neq$\\
    \hline
    $37$ & $\left<4\right>$ &
    $53$ & $\left<4\right>$ &
    $55$ & $\left<-1,4\right>$\\
    $61$ & $\left<4\right>$ &
    $65$ & $\left<2,12^2\right>$ &
    $65$ & $\left<4,12\right>$\\
    $65$ & $\left<4,24\right>$ &
    $89$ & $\left<9\right>$ &
    $101$ & $\left<4\right>$\\
     
    \hline

\caption*{}
\end{longtable}
\end{center}
\end{cor}

\begin{proof}
    For all these levels $N$ the only positive rank elliptic curves $E/\Q$ appearing in the decomposition of $J_0(N)$ have conductor equal to $N$. Additionally, all these curves $E$ are of (analytic) rank $1$. Let $f: X_0(N) \to E$ be a a strong Weil parametrization. Then \Cref{deltastrongweildegree} tells us that the composition $f\circ u_\Delta: X_\Delta(N)\to X_0(N) \to E$ is a strong Weil parametrization as well. The map $f\circ u_\Delta$  has degree $\deg f\cdot\#((\Z/N\Z)^\times/\Delta)$.
    
    For levels $N$ in the first list of this corollary we have $\deg f \geq5$ and for pairs $(N,\Delta)$ in this table we have $\deg f=2$ and $\#((\Z/N\Z)^\times/\Delta)\geq3$. The statement follows since every map to a positive rank elliptic curve has to factor through a strong Weil curve with positive rank.
\end{proof}

\subsection{Conductor strictly smaller then $N$}\label{section_condM}
Let $E'$ be an elliptic curve of conductor $<N$ occuring as an isogeny factor of $J_\Delta(N)$. In this case we have $\textup{Cond } E'=M\mid N$, $M<N$. 

Our goal is to rule out the existence of a degree $4$ map $X_\Delta(N) \to E'$ for the curves $X_\Delta(N)$ that do not occur in \Cref{mainthm} and for which the curve $X_0(N)$ has infinitely many quartic points (these levels $N$ are listed in \Cref{quarticthmx0n} - if $X_0(N)$ has only finitely many quartic points, we are immediately done). There are only four such levels $N$, namely $N\in\{74,86,111, 159\}$. For $N=74,111$ the only possible $E$ is $\lmfdbec{37}{a}{1}$, for $N=86$ it is $E=\lmfdbec{43}{a}{1}$, and for $N=159$ it is $E=\lmfdbec{53}{a}{1}$. Notice that these three curves are the only ones in their $\Q$-isogeny classes, and thus they do not admit any isogenies over $\Q$ except for the multiplication by $n$ maps.

Let us again consider a strong Weil curve $E/\Q$ for $X_0(M)$ in the isogeny class of $E'$. Note that in our four cases we have $E'=E$ since the isogeny classes in question contain only one element. But we will not make use of this since we want to present a method how to do the computation in the case where there are non-trivial isogenies over $\Q$. Let $\psi: E \to E'$ be an isogeny of minimal degree then we have the following diagram:

\begin{equation}\label{commutativediagram}
    X_\Delta(N)\xrightarrow{u_\Delta} X_0(N)\xrightarrow{\iota_{d,N,M}} X_0(M) \xrightarrow{f} E\xrightarrow{\psi} E'.
\end{equation}

Here $\iota_{d,N,M}$ is the degeneracy map. For every divisor $d\mid N/M$ we have a degeneracy map
\[\iota_{d,N,M}:X_0(N)\to X_0(M), (E,G)\mapsto (E/G[d], (G/G[d])[M]).\]
It is defined over $\Q$ and acts on $\tau \in \mathcal{H}^*$ in the extended upper half-plane as
\[\iota_{d,N,M}(\tau)=d\tau.\]

\Cref{prop:derickx-orlic-quadratic-form} tells us that the degree map $\deg: \Hom_\Q(X_\Delta(N),E') \to \Z$ can be extended to a positive definite quadratic form on $\Hom_\Q(J_\Delta(N),E') \cong \Z^n$ where $n$ is the number of positive divisors of $N/M$. If we can find the basis for $\Hom_\Q(J_\Delta(N),E)$ and the coefficients of this quadratic form, we will be able to determine all possible degrees of rational morphisms $X_\Delta(N)\to E'$.

\begin{prop}\label{deltashimura_condM<N}
    Let $(N,E)\in\{(74,\lmfdbec{37}{a}{1}),(86,\lmfdbec{43}{a}{1},2),(111,\lmfdbec{37}{a}{1}),(159,\lmfdbec{53}{a}{1})\}$. Denote the conductor of $E$ by $M$ and let us take $\{\pm1\}\subseteq\Delta\subseteq(\Z/N\Z)^\times$ and let $n$ be the number of divisors of $N/M$. 
    Let \[\xi_{E,N}=(f\circ\iota_{1,N,M},\ldots, f\circ\iota_{N/M,N,M}):X_0(N)\to E^n\] (we use the notation from \cite[Definition 3.8]{DerickxOrlic23}), then the composition \[u^*_\Delta\circ\xi^*_{E,N}:E^n\rightarrow J_0(N)\rightarrow J_\Delta(N)\] has trivial kernel. Therefore, the maps $f\circ\iota_{d,N,M}\circ u_\Delta$ form a basis for $\Hom_\Q(J_\Delta(N),E)$.
\end{prop}

\begin{proof}
     Let $A=\xi^*_{E,N}(E^n)\subseteq J_0(N)$. We computed that the Shimura subgroup $\Sigma(N)\cap A$ is trivial using the SageMath function \texttt{A.shimura-subgroup()}. Since $\Sigma(\Delta)\subseteq\Sigma(N)$, this implies that the group $\Sigma(\Delta)\cap A$ is also trivial. Together with the fact that $\xi^*_{E,N}$ is injective by \cite[Proposition 3.9]{DerickxOrlic23} (because $E$ is a strong Weil curve for $X_0(M)$), this proves the first part of the proposition. By \cite[Proof of Proposition 3.9]{DerickxOrlic23}, this means that the maps $f\circ\iota_{d,N,M}\circ u_\Delta$ form a basis for $\Hom_\Q(J_\Delta(N),E)$.
\end{proof}

A similar computation can be done for all levels $N$ and elliptic curves $E/\Q$ of conductor $M\mid N$. Similarly as in \Cref{shimurakernelmagma}, we expect that the kernel will be trivial for strong Weil elliptic curves of odd analytic rank.

The next result gives us a method to determine the  possible degrees of elements of $\Hom_\Q(X_\Delta(N),E')$ via the degree pairing $\left<\wideunder,\wideunder\right>$ defined in \Cref{degreepairing-section}.

\begin{prop}\label{quadraticform_coefficients}
    Suppose that we have a setup as in Diagram \ref{commutativediagram} and that the map $E^n\to J_\Delta(N)$ from \Cref{deltashimura_condM<N} has trivial kernel. If $\psi:E\to E'$ is a rational isogeny of minimal degree, then the quadratic form \[\deg \psi \cdot\#((\Z/N\Z)^\times/\Delta)\sum_{i\mid N/M}\sum_{j\mid N/M}x_ix_j\left<f\circ \iota_{i,N,M},f\circ \iota_{j,N,M}\right>\] represents the degrees of all rational morphisms $h:X_\Delta(N)\to E'$.
\end{prop}
\begin{proof}
    By \cite[Proof of Theorem 1.12]{DerickxOrlic23}, the quadratic form \[\sum_{i\mid N/M}\sum_{j\mid N/M}x_ix_j\left<f\circ \iota_{i,N,M},f\circ \iota_{j,N,M}\right>\] represents the degrees in $\Hom_\Q(X_0(N),E)$. \Cref{lem:degree_pairing_covering_curve} tells us that the degree pairing on $\Hom_\Q(X_\Delta(N),E)$ is the degree pairing on $\Hom_\Q(X_0(N),E)$ multiplied by $\deg u_\Delta=\#((\Z/N\Z)^\times/\Delta)$ and \Cref{lem:degree_pairing_isogeny} tells us that the degree pairing on $\Hom_\Q(X_\Delta(N),E')$ is the degree pairing on $\Hom_\Q(X_\Delta(N),E)$ multiplied by $\deg\psi$.
\end{proof}
Note that the maps $f\circ\iota_{i,N,M}$ go from $X_0(N)$ to $E$. This result is important because the ratio of the coefficients of the quadratic forms for $\Hom_\Q(J_\Delta(N),E)$ and $\Hom_\Q(J_0(N),E)$ is constant (it does not depend on $i$ and $j$). Therefore, we can compute the quadratic form for $\Hom_\Q(J_0(N),E)$ (which we know how to do) and scale it to obtain the quadratic form for $\Hom_\Q(J_\Delta(N),E)$.

The following result computes the coefficients of the quadratic form for $X_0(N)\to E$.

\begin{thm}[{\cite[Theorem 2.13]{DerickxOrlic23}}]\label{pairingcomputation}
Let $E$ be an elliptic curve of conductor $M$ with the corresponding newform $\sum_{n=1}^\infty a_nq^n$ and let $f:X_0(M)\to E$ be the modular parametrization of $E$. Assume that $\frac{N}{M}$ is either squarefree or coprime to $M$ and let $d_1$ and $d_2$ be divisors of $\frac{N}{M}$. We write $\gcd$ instead of $\gcd(d_1,d_2)$ and $\textup{lcm}$ instead of $\textup{lcm}(d_1,d_2)$ for simplicity. If we define
$$a=\left(\sum_{m^2\mid (d_1/\gcd)} \mu(m)a_{d_1/(\gcd m^2)}\right)\left(\sum_{m^2\mid (d_2/\gcd)} \mu(m)a_{d_2/(\gcd m^2)}\right),$$
where $\mu$ is the M\"obius function (when $\frac{d_1d_2}{\gcd^2}$ is squarefree, $a$ is equal to $a_{d_1d_2/\gcd^2}$), then
$$\left<f\circ\iota_{d_1,N,M},f\circ\iota_{d_2,N,M}\right>=\left[a\cdot \frac{\psi\left(N\right)}{\psi\left(\frac{M\textup{lcm}}{\gcd}\right)}\cdot \deg f \right].$$
Here $\psi(N)=N\prod_{q\mid N}(1+\frac{1}{q})$ is the index of the congruence subgroup $\Gamma_0(N)$ in $\SL_2(\Z)$.
\end{thm}

\begin{remark}\label{quadraticformdivisiblebydef_remark}
    Even if we do not have that $\frac{N}{M}$ is either squarefree or coprime to $M$, we can still prove that the coefficient is divisible by $\deg f$ because \[\left<f\circ\iota_{d_1,N,M},f\circ\iota_{d_2,N,M}\right>=[\deg f]\circ\left<\iota_{d_1,N,M},\iota_{d_2,N,M}\right>,\] see \cite[Start of the proof of Theorem 2.13]{DerickxOrlic23} for more details.
\end{remark}

\Cref{pairingcomputation} is useful because all items on the right-hand side are easily computable, and in fact already have been computed for all elliptic curves of conductor $\leq500,000$ and $\textup{lcm}(d_1,d_2)/\gcd(d_1,d_2) \leq 1,000$. This data is available in the LMFDB \cite{lmfdb}.

\begin{prop}\label{propquadraticforms_tetraelliptic}
    The modular curve $X_\Delta(N)$ is not positive rank tetraelliptic for $N\in\{74,86,111,159\}$ and all $\{\pm1\}\subseteq \Delta\subsetneq (\Z/N\Z)^\times$.

\begin{center}
\begin{longtable}{|c|c|c|c|}

\hline
\addtocounter{table}{-1}
    $N$ & $M$ & $E$ & \textup{quad. form for} $\Hom_\Q(J_0(N), E)$\\
    \hline

    $74$ & $37$ & $\lmfdbec{37}{a}{1}$ & $6x^2-8xy+6y^2$\\
    $86$ & $43$ & $\lmfdbec{43}{a}{1}$ & $6x^2-8xy+6y^2$\\
    $111$ & $37$ & $\lmfdbec{37}{a}{1}$ & $8x^2-12xy+8y^2$\\
    $159$ & $53$ & $\lmfdbec{53}{a}{1}$ & $8x^2-12xy+8y^2$\\
     
    \hline

\caption*{}
\end{longtable}
\end{center}
\end{prop}

\begin{proof}
    These three elliptic curves are the only ones in their isogeny classes, therefore $\deg \psi=1$. Since $N/M$ is equal to $2$ or $3$, we have $n=2$ in all four cases and we can use \Cref{pairingcomputation} to compute the quadratic forms for $\Hom_\Q(J_0(N), E)$ (see \cite[Examples 1,2,3]{DerickxOrlic23} for more details). 
    
    By \Cref{deltashimura_condM<N}, the induced map $E^2\to J_0(N)\to J_\Delta(N)$ has trivial kernel, which means that we can use \Cref{quadraticform_coefficients} to obtain the quadratic form for $\Hom_\Q(J_\Delta(N), E)$. We just need to multiply the quadratic form for $\Hom_\Q(J_0(N), E)$ by $\#((\Z/N\Z)^\times/\Delta)$.
    
    We can easily check that the least positive integer attained by the quadratic forms in the table is $4$. Since $\#((\Z/N\Z)^\times/\Delta)\geq2$, this proves that there is no degree $4$ rational morphism from $X_\Delta(N)$ to $E$.
\end{proof}

\section{Proof of the main theorem}\label{sectionmainthm}

\begin{proof}[Proof of \Cref{mainthm}]
    For curves $X_\Delta(N)$ listed in the theorem, in \Cref{infinitelyquarticsection} we find a degree $4$ rational morphism to $\PP^1$ or a positive rank elliptic curve, or there is a degree $2$ projection to a curve $X_0(N)$ with infinitely many quadratic points. The pullbacks of rational or quadratic points are a source of infinitely many quartic points on $X_\Delta(N)$.

    All other intermediate curves $X_\Delta(N)$ have genus $g\geq8$ and \cite[Theorem 1.3]{Orlic2024Intermediate} tells us that all these curves have $\Q$-gonality greater than $4$. We eliminate all but finitely many levels $N$ in \Cref{preliminariessection}.
    
    From \Cref{degree4morphismcor} it follows that it is enough to prove that these curves $X_\Delta(N)$ do not admit a degree $4$ rational morphism to a positive rank elliptic curve, which is done in \Cref{deltatetraelliptic_rank0}, \Cref{tetraelliptic_condN}, and \Cref{propquadraticforms_tetraelliptic}.
\end{proof}

\bibliographystyle{siam}
\bibliography{bibliography}
\end{document}